\documentclass[reqno, a4paper, 10pt]{amsart}
\usepackage{amssymb,url, color, mathrsfs}
\usepackage[colorlinks=true, bookmarks=true, pdfstartview=FitH, pagebackref=true, linktocpage=true]{hyperref}
\usepackage[short,nodayofweek]{datetime}
\usepackage[pagewise,running,mathlines,displaymath, switch]{lineno}
\usepackage{times}
\usepackage{indentfirst, tabularx}
\usepackage[table]{xcolor}
\usepackage{float, hhline, colortbl}
\usepackage{graphicx}
\usepackage{arydshln} 
\usepackage{setspace}

\hypersetup{
pdftitle={Supercritical MT} 
pdfauthor={ABC},
colorlinks = true,
linkcolor = magenta,
citecolor = blue,
}

\theoremstyle{plain}
\newtheorem{theorem}{Theorem}[section]

\newtheorem{lemma}[theorem]{Lemma}

\theoremstyle{definition}

\theoremstyle{remark}

\def\R{\mathbb R}

\def\al{\alpha}
\def\om{\omega}
\def\ga{\gamma}
\def\de{\delta}
\def\De{\Delta} 

\def\vphi{\varphi}
\def\ep{\epsilon}
\def\na{\nabla}
\def\pa{\partial}
\def\lt{\left}
\def\rt{\right}

\def\H1r{H^1_{0,{\rm rad}}(B)}

\def\Wn{W^{1,n}_{0}(B)}
\def\Wnr{W^{1,n}_{0,{\rm rad}}(B)}
\def\Wmp{W^{m,p}_{0}(B)}
\def\Wmpr{W^{m,p}_{0,{\rm rad}}(B)}
\def\Wm2r{W^{m,2}_{0,{\rm rad}}(B)}
\def\Cr{C^\infty_{0,{\rm rad}}(B)}

\def\MT{\mathsf{MT}_n}

\def\s2{2^\star}
\def\2ms{2_m^\star}

\numberwithin{equation}{section}

\makeatletter
\def\@cite#1#2{[\textbf{#1}\if@tempswa, #2\fi]}
\makeatother

\title{Supercritical }
\def\cfac#1{\ifmmode\setbox7\hbox{$\accent"5E#1$}\else\setbox7\hbox{\accent"5E#1}\penalty 10000\relax\fi\raise 1\ht7\hbox{\lower1.05ex\hbox to 1\wd7{\hss\accent"13\hss}}\penalty 10000\hskip-1\wd7\penalty 10000\box7 }

\title[Supercritical Moser--Trudinger inequalities]{Supercritical Moser--Trudinger inequalities and related elliptic problems}

\author[Q.A. Ng\^o]{Qu\cfac{o}c Anh Ng\^o}

\address[Q.A. Ng\^o]{Department of Mathematics\\
College of Science, Vi\^{e}t Nam National University\\
H\`{a} N\^{o}i, Vi\^{e}t Nam.}
\email{\href{mailto: Q.A. Ng\^o <nqanh@vnu.edu.vn>}{nqanh@vnu.edu.vn}}
\email{\href{mailto: Q.A. Ng\^o <bookworm\_vn@yahoo.com>}{bookworm\_vn@yahoo.com}}

\author[V.H. Nguyen]{Van Hoang Nguyen}

\address[V.H. Nguyen]{Institute of Mathematics\\
Vietnam Academy of Science and Technology\\
Hanoi, Vietnam.}

\email{\href{mailto: V.H. Nguyen <vanhoang0610@yahoo.com>}{vanhoang0610@yahoo.com}}
\email{\href{mailto: V.H. Nguyen <nvhoang@math.ac.vn>}{nvhoang@math.ac.vn}}

\allowdisplaybreaks

\begin{document}

\begin{abstract}
Given $\alpha >0$, we establish the following two supercritical Moser--Trudinger inequalities
\[
\sup\limits_{\begin{subarray}{c} 
u \in \Wnr: 
\int_B |\nabla u|^n dx \leq 1
\end{subarray}} 
\int_B \exp\big( (\alpha_n + |x|^\al) |u|^{\frac{n}{n-1}} \big) dx < +\infty
\]
and
\[
\sup\limits_{\begin{subarray}{c} 
u\in \Wnr: 
\int_B |\nabla u|^n dx \leq 1
\end{subarray}}
\int_B \exp\big( \alpha_n |u|^{\frac{n}{n-1} + |x|^\alpha} \big) dx < +\infty,
\]
where $\Wnr$ is the usual Sobolev spaces of radially symmetric functions on $B$ in $\mathbb R^n$ with $n\geq 2$. Without restricting to the class of functions $\Wnr$, we should emphasize that the above inequalities fail in $\Wn$. Questions concerning the sharpness of the above inequalities as well as the existence of the optimal functions are also studied. To illustrate the finding, an application to a class of boundary value problems on balls is presented. This is the second part in a set of our works concerning functional inequalities in the supercritical regime.
\end{abstract}

\date{\bf \today \; at \currenttime}

\subjclass[2000]{46E35, 26D10, 35A15}

\keywords{supercritical Moser--Trudinger inequality; sharp constant; optimizer; elliptic problems with critical growth}

\maketitle

\section{Introduction}

This is the second part in a set of our works concerning functional inequalities in the supercritical regime. Previously in \cite{NN19}, given $n \geq 3$, together with \cite{doO} we have shown that there is a continuous embedding
\begin{equation}\label{eq:Sobolev2}
\Wm2r \hookrightarrow L_{2_m^\star + |x|^\alpha} (B)
\end{equation}
with $1 \leq m < n/2$, $2_m^* = 2n/(n-2m)$, and $\alpha \geq 0$. In the embedding \eqref{eq:Sobolev2}, $B$ denotes the unit ball in $\R^n$ and $\Wmpr$ is a subspace of the Sobolev space $\Wmp$, which consists of only radially symmetric functions. Here by `radially symmetric' we mean `radially symmetric about the origin'. If we further denote by $\Cr$ the class of compactly supported, smooth, radially symmetric functions in $B$, then the space $\Wmpr$ can also be defined as the completion of $\Cr$ under the norm
\[
\|u\|_{\Wmpr} = \Big( \int_B |\nabla^m u|^p dx \Big) ^{1/p},
\]
where, for an integer $m \geq 1$, we use the following notation
\[
\nabla^m = \begin{cases}
\De^{m/2} &\mbox{if $m$ is even},\\
\nabla \De^{(m-1)/2} &\mbox{if $m$ is odd}.
\end{cases}
\]
When $m=1$ and $\alpha=0$, the embedding \eqref{eq:Sobolev2} belongs to a wider class of inequalities, which state that the following embedding 
\begin{equation}\label{eq:Sobolevp}
W_0^{1,p} (B) \hookrightarrow L_{2n/(n-p)} (B)
\end{equation}
holds whenever $p<n$. In general, one cannot take the limit as $q \nearrow n$, that is, the following embedding
\[
\Wn \hookrightarrow L_\infty (B)
\]
is no longer available. Instead, Trudinger's inequality \eqref{eq:Moser--Trudinger} below provides us a perfect replacement, namely, there holds
\[
\Wn \hookrightarrow e^{L_{n/(n-1)}} (B).
\]
The choice of the space $e^{L_{n/(n-1)}} (B)$ stems from the fact that given any $u \in \Wn$, there holds
\[
 \int_\Omega \exp\big(\gamma |u|^\frac{n}{n-1} \big) dx < +\infty 
\]
for any $\gamma \geq 0$. Historically, Trudinger's inequality on bounded domains was established independently by Yudovi\v c \cite{Y1961}, Poho\v zaev \cite{P1965}, and Trudinger \cite{T1967}. It is stated that there is some constant $\gamma > 0$ such that
\[
\sup\limits_{\begin{subarray}{c} 
u \in W_0^{1,n} (B):
 \int_\Omega |\nabla u|^n dx \leq 1
\end{subarray}}
 \int_\Omega \exp\big(\gamma |u|^\frac{n}{n-1} \big) dx < +\infty 
\]
Later, by sharpening Trudinger's inequality, Moser \cite{M1970} proved that there exists a dimensional constant $\alpha_n > 0$ such that the above inequality holds for any $\gamma \leq \alpha_n$, namely
\begin{equation}\label{eq:Moser--Trudinger}
\sup\limits_{\begin{subarray}{c} 
u \in W_0^{1,n} (B):
 \int_\Omega |\nabla u|^n dx \leq 1
\end{subarray}}
 \int_\Omega \exp\big(\gamma |u|^\frac{n}{n-1} \big) dx < +\infty 
\end{equation}
holds for any $\gamma\leqslant \alpha_n$ and for any bounded domain $\Omega$ in $\R^n$. Remarkably, Moser was able to compute the constant $\alpha_n$ precisely, that is
\[
\alpha_n = n^{n/(n-1)} \Omega_n^{1/(n-1)},
\] 
where $\Omega_n$ denotes the volume of the unit ball $\mathbb B^n$ in $\R^n$. If we denote by $\omega_n$ the volume of the unit sphere $\mathbb S^n$ in $\R^{n+1}$, then
\[
\alpha_n = n \omega_{n-1}^{1/(n-1)}.
\]
Apparently, Inequality \eqref{eq:Moser--Trudinger}, also known as the Moser--Trudinger inequality, can be thought of as a limiting case of the well-known Sobolev inequality \eqref{eq:Sobolevp}. Since Inequality \eqref{eq:Moser--Trudinger} and its variants have many applications in many aspects of analysis, generalizing of \eqref{eq:Moser--Trudinger} has already been a hot research topic and a huge set of works have already been written within the last two decades.

Back to Inequality \eqref{eq:Moser--Trudinger}, the following demonstrates that \eqref{eq:Moser--Trudinger} is the best possible if one only works on the class of functions in $\Wn$. First, the constant $\alpha_n$ in \eqref{eq:Moser--Trudinger} is sharp in the sense that if $\gamma > \alpha_n$, then there exists a sequence of functions $(u_j)_j$ in $\Wn$ with $ \int_\Omega |\nabla u|^n dx =1$ and with 
\[
 \int_\Omega \exp\big(\gamma |u_j|^\frac{n}{n-1} \big) dx \nearrow +\infty
\]
as $j \to +\infty$, which implies that the supremum in \eqref{eq:Moser--Trudinger} becomes infinity. Furthermore, the exponent $n/(n-1)$ in \eqref{eq:Moser--Trudinger} is also sharp because this is the maximal growth. There are examples of functions $u$ such that the left hand side of \eqref{eq:Moser--Trudinger} becomes infinite if one replaces either $\alpha_n$ or $n/(n-1)$ by any greater number.

Following the strategy shown in \cite{NN19, doO}, in this part of the program, we are interested in the classical Moser--Trudinger inequality \eqref{eq:Moser--Trudinger} in the supercritical regime. To be more precise, we aim to improve the threshold $\alpha_n$ as well as the exponent $n/(n-1)$. 
Since the Moser--Trudinger inequality \eqref{eq:Moser--Trudinger} is monotone increasing with respect to $\gamma$, let us focus on the following
\begin{equation}\label{eq:MT}
\sup\limits_{\begin{subarray}{c} 
u \in \Wn:
 \int_\Omega |\nabla u|^n dx \leq 1
\end{subarray}}
 \int_\Omega \exp\big(\alpha_n |u|^\frac{n}{n-1} \big) dx < +\infty
\end{equation}
and only call \eqref{eq:MT} the Moser--Trudinger inequality. We can also denote by $\MT$ the left hand side of \eqref{eq:MT}, namely
\[
\MT = \sup_{u\in \Wnr, \int_B |\nabla u|^n dx \leq 1} \int_B \exp\big(\alpha_n |u|^{\frac{n}{n-1}} \big) dx,
\]
which is positive and finite. 

To quickly identify our improvement for the classical Moser--Trudinger inequality, let us jump into our first result, which consists of two supercritical Moser--Trudinger type inequalities:

\begin{theorem}\label{SuperMT}
Let $\alpha >0$ and $n \geq 2$. Then we have
\begin{equation}\label{eq:SuperMT1}
\sup\limits_{\begin{subarray}{c} 
u \in \Wnr:
\int_B |\nabla u|^n dx \leq 1
\end{subarray}} 
\int_B \exp\big( (\alpha_n + |x|^\al) |u|^{\frac{n}{n-1}} \big) dx < +\infty
\end{equation}
and
\begin{equation}\label{eq:SuperMT2}
\sup\limits_{\begin{subarray}{c} 
u\in \Wnr:
\int_B |\nabla u|^n dx \leq 1
\end{subarray}}
\int_B \exp\big( \alpha_n |u|^{\frac{n}{n-1} + |x|^\alpha} \big) dx < +\infty,
\end{equation}
where $\alpha_n = n \omega_{n-1}^{1/(n-1)}$.
\end{theorem}

In the first glimpse, the improved inequalities \eqref{eq:SuperMT1} and \eqref{eq:SuperMT2} seem too strong to be true. Indeed, if we do not restrict ourselves to the class of functions $\Wnr$ and only consider functions in $\Wn$, then Theorem \ref{SuperMT} is \textit{no longer} true. The idea is to select a region of $B$ far from the origin and also far from the boundary where the extra term $|x|^\alpha$ has some role. Indeed, fix some $x_0 \in B$ with $|x_0|=1/2$ and choose $r = 1/4$. Since any function in $w \in W_0^{1,n}(B_r(x_0))$ with $\int_{B_r(x_0)} |\nabla w|^n dx \leq 1$ also belongs to $W_0^{1,n}(B)$ with $\int_B |\nabla w|^n dx \leq 1$, we deduce that
\begin{align*}
\sup\limits_{\begin{subarray}{c} 
u\in \Wn : \\\int_B |\nabla u|^n dx \leq 1
\end{subarray}} 
\int_B \exp\big( &(\alpha_n + |x|^\al) |u|^{\frac{n}{n-1}} \big) dx \\
&\geq 
\sup\limits_{\begin{subarray}{c} 
w\in W^{1,n}_{0}(B_r(x_0)): \\ \int_{B_r(x_0)} |\nabla w|^n dx \leq 1
\end{subarray}} 
\int_{B_r(x_0)} \exp\big( (\alpha_n + |x|^\al) |w|^{\frac{n}{n-1}} \big) dx.
\end{align*}
However, within $B_r(x_0)$, we always have $|x| \geq 1/4$. Hence
\begin{align*}
\sup\limits_{\begin{subarray}{c} 
u\in \Wn:\\ \int_B |\nabla u|^n dx \leq 1
\end{subarray}} 
& \int_B \exp\big( (\alpha_n + |x|^\al) |u|^{\frac{n}{n-1}} \big) dx \\
& \geq 
\sup\limits_{\begin{subarray}{c} 
w\in W^{1,n}_{0}(B_r(x_0)): \\ \int_{B_r(x_0)} |\nabla w|^n dx \leq 1
\end{subarray}} 
\int_B \exp\Big( \big[\alpha_n + \big(\frac 14 \big)^\alpha \big] |w|^{\frac{n}{n-1}} \Big) dx
= +\infty,
\end{align*}
where the last assertion comes from the sharpness of the classical Moser--Trudinger inequality \eqref{eq:MT}. A similar argument works for \eqref{eq:SuperMT2}. Back to the validity of our inequalities  \eqref{eq:SuperMT1} and \eqref{eq:SuperMT2}, by comparing to existing results in a similar superciritical regime in the literature, see e.g. \cite{CR2015, CR2015-NA}, one would speculate that these inequalities can be true.

Before going further, we note that going beyond the thresholds $\alpha_n$ and $n/(n-1)$ in the classical Moser--Trudinger inequality is an extensive topic of research in the last two decades. A prior to the present work, a large number of works focus on the sharp constant $\alpha_n$ by perturbing it by a small number heavily depending on some norm of $u$. Among the works related to this direction, we can recall the following interesting inequality
\begin{equation}\label{eq:MT-Yang}
\sup\limits_{\begin{subarray}{c} 
u \in \Wn:
\int_B |\nabla u|^n dx = 1
\end{subarray}} 
\int_B \exp\big( \alpha_n [ 1+ \gamma \|u\|_n^n ]^\frac{1}{n-1} |u|^{\frac{n}{n-1}} \big) dx < +\infty
\end{equation}
for $0 \leq \gamma < \lambda_1 (B)$, proved by Adimurthi and Druet \cite{AD04} for $n=2$ and by Yang \cite{Yang06} for $n \geq 3$. Here $\| \cdot \|_n$ denotes the $L^n$-norm with respect to the Lebesgue measure and $\lambda_1 (B)$ is the first eigenvalue of the $n$-Laplacian with Dirichlet boundary condition in $B$. We note that \eqref{eq:MT-Yang} does not violate \eqref{eq:MT}, it does give more precise information than inequality \eqref{eq:MT} in the sense described below. Suppose that $(u_j)_j $ is a maximizing sequence for the left hand side of \eqref{eq:MT-Yang}. Since $\|\nabla u_j\|_n =1$, there is some $u_\infty \in \Wn$ such that $u_j \rightharpoonup u_\infty$ weakly in $\Wn$. If $u_\infty \not\equiv 0$, the obvious inequality
\[
1+ \gamma \|u_\infty\|_n^n \leq \frac 1{1-\|\nabla u_\infty\|_n^n}
\]
implies that \eqref{eq:MT} is a consequence of the well-known concentration-compactness principle of Lions; see \cite{Lions1985}. If $u_\infty \equiv 0$, the compact embedding $\Wn \hookrightarrow L_n(B)$ implies that $\|u_\infty\|_n = 0$, yielding that the parameter $\gamma$ has no effect. 

In addition to Theorem \ref{SuperMT} above, in the following result, we shall show that in our supercritical regime the threshold $\alpha_n$ is sharp.

\begin{lemma}\label{Sharpness}
Let $\alpha >0$ and $n \geq 2$. Then the threshold $\alpha_n$ in both \eqref{eq:SuperMT1} and \eqref{eq:SuperMT2} is sharp in the sense that if we replace $\alpha_n$ by any number $\gamma > \alpha_n$ then the supremum in both \eqref{eq:SuperMT1} and \eqref{eq:SuperMT2} becomes infinity, namely
\[
\sup\limits_{\begin{subarray}{c} 
u \in \Wnr:
\int_B |\nabla u|^n dx \leq 1
\end{subarray}} 
\int_B \exp\big( (\gamma + |x|^\al) |u|^{\frac{n}{n-1}} \big) dx = +\infty
\]
and
\[
\sup\limits_{\begin{subarray}{c} 
u\in \Wnr:
\int_B |\nabla u|^n dx \leq 1
\end{subarray}}
\int_B \exp\big( \gamma |u|^{\frac{n}{n-1} + |x|^\alpha} \big) dx = +\infty,
\]
whenever $\gamma > \alpha_n$.
\end{lemma}

For simplicity and later uses, for each $\alpha >0$, let us simply denote
\begin{align*}
\MT^1(\al) &= \sup\limits_{\begin{subarray}{c} 
u\in \Wnr : 
\int_B |\nabla u|^n dx \leq 1
\end{subarray}} 
\int_B \exp\big( (\alpha_n + |x|^\al) |u|^{\frac{n}{n-1}} \big) dx,
\intertext{and}
\MT^2(\al) &=\sup\limits_{\begin{subarray}{c} 
u\in \Wnr : 
\int_B |\nabla u|^n dx \leq 1
\end{subarray}} 
\int_B \exp\big( \alpha_n |u|^{\frac{n}{n-1} + |x|^\alpha} \big) dx.
\end{align*}
Clearly, in view of Theorem \ref{SuperMT}, the sharp constants $\MT^i(\al)$ with $i=1,2$ are finite. As routine, now we turn out attention to the attainability of the sharp constants $\MT^i$ with $i=1,2$. Identifying optimal functions for functional inequalities is always a delicate issue. For example, although the supercritical Moser--Trudinger inequality \eqref{eq:MT-Yang} is valid for any $\gamma \in [0, \lambda_1 (B))$, optimal functions for \eqref{eq:MT-Yang} can only exist if $\gamma$ stays a way from $\lambda_1(B)$ except $n \geq 3$; see \cite{Yang06, MT2019}. In our supercritical regime, we are successful in proving that optimal functions exist in the full range of the parameter $\alpha$ regardless of the size of $n$. Our next result is the following.

\begin{theorem}\label{Attain}
Let $\alpha >0$ and $n \geq 2$. Then the sharp constants $\MT^i (\al)$ with $i=1,2$ are attained in $\Wnr$.
\end{theorem}

Let us now briefly comment of the idea of the proof of Theorem \ref{Attain}. To obtain the attainability of the constants $\MT^i (\al)$, $i=1,2$, it is routine to examine the maximizing sequence for $\MT^i (\al)$, $i=1,2$. Two ingredients in our proof is a concentration-compactness principle of Lions type established in Lemma \ref{CC} and a lower bound for $\MT^i (\al)$ established in Lemmas \ref{Strict} and \ref{Strict*}.

Finally, we are interested in applications of our supercrititcal inequalities \eqref{eq:SuperMT1} and \eqref{eq:SuperMT2}. Among many applications of the classical Moser--Trudinger inequality, we choose the following typical problem
\begin{equation}\label{eq:Super}
\left\{
\begin{aligned}
-\Delta_n u &= f(x,u) && \mbox{ in } B,\\
u&> 0 &&\mbox{ in } B,\\
 u &= 0 && \mbox{ on } \pa B,
\end{aligned}
\right.
\end{equation}
where the nonlinearity $f$ has the maximal growth on $u$. To relate the maximal growth on $u$ and our new supercritical Moser--Trudinger inequalities, we propose the following new terminology: The function $f$ is said to have the \textit{critical growth} $\alpha_0$ on $B$ if for any $\beta>\alpha_0$ there holds
\begin{equation}\label{CG>}
\lim_{|t| \to +\infty} \frac{|f(x,t)|}{\exp (\beta |t|^{\frac{n}{n-1} + |x|^\alpha})} = 0
\end{equation}
uniformly on $x \in B$ and for any $\beta < \alpha_0$ there holds
\begin{equation}\label{CG<}
\lim_{|t| \to +\infty} \frac{|f(x,t)|}{\exp (\beta |t|^{\frac{n}{n-1} + |x|^\alpha})} = +\infty
\end{equation}
uniformly on $x \in B$. Clearly, the usual critical growth used in many works prior to this work differs from ours by the exponent $|x|^\alpha$. Furthermore, in view of \eqref{CG>} and \eqref{CG<}, one can define a similar critical growth by replacing the exponent $\beta |t|^{\frac{n}{n-1} + |x|^\alpha}$ by the exponent $(\beta + |x|^\alpha ) |t|^{\frac{n}{n-1} }$. However, we do not treat this case in the present paper and leave it for interested readers. 

Inspired by many works, for instance, \cite{doO1996,FMR,FOR02}, we are going to impose the following conditions on the nonlinearity $f$:
\begin{enumerate}
\item[($F_1$)] $f : \overline B \times \R \to \R$ is continuous and \textit{radially symmetric} in the first variable, namely
\[
f(x,t)=f(y,t)
\]
whenever $|x|=|y|$.

\item[($F_2$)] There exist $R>0$ and $M>0$ such that
\[
0 < F(x,t) = \int_0^t f(x,s)ds \leq M f(x,t)
\]
for any $t \geq R$ and any $x \in B$.

\item[($F_3$)] There hold $f(x,t) \geq 0$ for any $(x,t) \in B \times \R$ and $f(x,0) = 0$ for any $x \in B$.

\item[($F_4$)] There holds
\[
\limsup_{t \searrow 0} \frac{nF(x,t)}{t^n} < \lambda_1 (B)
\]
uniformly on $B$.

\item[($F_5$)] There holds
\[
\lim_{t \nearrow +\infty} \frac{t f(x,t) }{\exp (\alpha_0 t^{\frac{n}{n-1} })} \geq \beta_0 >  \frac{n^n}{\al_0^{n-1} e^{1+ \cdots + \frac1{n-1}}}
\]
uniformly on $B$. 
\end{enumerate}

Apparently, a nonlinearity $f$ having critical growth in the sense of \eqref{CG>} and \eqref{CG<} behaves like $\exp (\alpha_0 |t|^{\frac n{n-1} + |x|^\alpha})$, which decays to infinity no slower than $\exp (\alpha_0 |t|^{\frac n{n-1}})$. Therefore, our condition ($F_5$) is somehow weaker than the following condition:
\begin{enumerate}
\item[($F'_5$)] There holds
\[
\lim_{t \nearrow +\infty} \frac{t f(x,t) }{\exp (\alpha_0 t^{\frac{n}{n-1} + |x|^\alpha})} \geq \beta_0 > \frac{n^n}{\al_0^{n-1} e^{1+ \cdots + \frac1{n-1}}}
\]
uniformly on $B$.
\end{enumerate}
We also take this chance to notice that our lower bound for $\beta_0$ in $(F_5)$ is also weaker than the existing hypothesis; see \cite[condition $(F_5)$]{doO1996}. Given $\alpha>0$ and consider the following function
\[
f(x,t) = t^{\frac{1}{n-1} + |x|^\alpha}
\left( 
\begin{aligned}
& \exp \big( \alpha_0 t^{\frac n{n-1} + |x|^{\alpha}} \big) - \sum_{i=0}^{n-3} \frac{\alpha_0^i}{i!} t^{(\frac n{n-1} + |x|^\alpha)i} \\
& - c \frac{\al_0^{n-2}}{(n-2)!} t^{(\frac n{n-1} + |x|^\alpha)(n-2)}
\end{aligned}
\right).
\]
Clearly, the above function $f$ with $c$ closed to $1$ satisfies all five conditions $(F_1)$--$(F_5)$.

Now we state our existence result for problem \eqref{eq:Super}.

\begin{theorem}\label{thmApplication}
Let $n \geq 2$. Assume that $f$ has critical growth $\alpha_0$ on $B$ and satisfies the five conditions ($F_1$)--($F_5$). Then there exists a positive $C^1$-solution to the supercritical problem \eqref{eq:Super}.
\end{theorem}

As can be easily notified, compared with other existing hypotheses on similar problems, there is an extra hypothesis on $f$, namely, $f$ is assumed to be radial on the first variable. We would like to comment that somehow this extra requirement is natural. Since we are in the superciritical case, it is expected to look for solutions in the space $\Wnr$. In this sense, because
\[
\Delta_n f = r^{1-n} (r^{n-1}|f'|^{n-2} f')'
\]
 for any radially symmetric function $f$, we know that $n$-Laplacian of a radially symmetric $C^2$-function is again radially symmetric. As such, if the problem \eqref{eq:Super} has a radially symmetric solution, the nonlinearity $f$ must be radially symmetric. In addition to the above fact, it is worth emphasizing that due to the radial symmetry of $f$ in the first variable, we are able to prove that the solution found in $\Wnr$ is strictly positive in $B$.

The present paper is organized as follows.

\tableofcontents


\section{Two supercritical Moser--Trudinger inequalities: Proof of Theorem \ref{SuperMT} and Lemma \ref{Sharpness}}

This section is devoted to a proof of Theorem \ref{SuperMT}. In fact, we shall prove a more general result which covers Theorem \ref{SuperMT} as a special case. The following is our main result of this section.

\begin{theorem}\label{SuperMT*}
Let $f:[0,1) \to [0,\infty)$ be a continuous function such that 
\begin{enumerate}
\item[($f_1$)] $f(0) = 0$ and $f(r) >0$ for $r >0$;
\item[($f_2$)] there exists some $c > 0$ such that
\[
f(r) \leq \frac c{-\log r}
\] 
for $r$ near $0$.
\item[($f_3$)] there exists some $\gamma \in (0,1)$ such that
\[
f(r) \leq \gamma \frac{\alpha_n}{n} \frac{\log (1-r)}{\log r}
\] 
for $r$ near $1$.
\end{enumerate}
Then, there holds
\begin{equation}\label{eq:SuperMT1*}
\sup\limits_{\begin{subarray}{c} 
u\in \Wnr : 
 \int_B |\nabla u|^n dx \leq 1
\end{subarray}} 
 \int_B \exp \big( (\alpha_n + f(|x|)) |u|^{\frac{n}{n-1}} \big) dx < +\infty.
\end{equation}
Moreover, if the condition ($f_2$) is replaced by the following condition ($f_2'$),
\begin{enumerate}
\item[($f_2'$)] there exists some $c > 0$ and some $\gamma>2$ such that
\[
f(r) \leq \frac c{(-\log r)^{\gamma}}
\] 
for $r$ near $0$,
\end{enumerate}
then there holds
\begin{equation}\label{eq:SuperMT2*}
\sup\limits_{\begin{subarray}{c} 
u\in \Wnr : 
 \int_B |\nabla u|^n dx \leq 1
\end{subarray}} 
 \int_B \exp \big( \alpha_n |u|^{\frac{n}{n-1} + f(|x|)} \big) dx < +\infty.
\end{equation}
\end{theorem}

\begin{proof}
Let $u \in \Wnr$ with $\|\nabla u\|_{L^n(B)} \leq 1$. Since $u$ is radially symmetric about the origin, we should emphasize that there is no difference if we write $u(|x|)$ instead of $u(x)$. By density, we may assume that $u \in \Cr$, which immediately yieds
\[
\int_0^1 |u'(s)|^n s^{n-1}ds \leq \om_{n-1}^{-1}. 
\]
Clearly, we always have
\[
u(r) = -\int_r^1 u'(s) ds = -\int_r^1 u'(s) s^{\frac {n-1}n} s^{-\frac {n-1}n} ds.
\]
By H\"older's inequality, we get the following pointwise estimate
\begin{align}\label{eq:growthu}
|u(r)| &\leq \Big(\int_r^1 |u'(s)|^n s^{n-1}ds\Big)^{1/n} \Big(\int_r^1 s^{-1} ds\Big)^{\frac{n-1}n}\notag\\
&\leq \om_{n-1}^{-1/n} \Big(\om_{n-1}\int_0^1 |u'(s)|^n s^{n-1}ds\Big)^{1/n} \big(-\log r\big)^{\frac{n-1}n}\notag\\
& \leq \Big(\frac{n}{\alpha_n}\Big)^{\frac{n-1}n} \big(-\log r\big)^{\frac{n-1}n}
\end{align}
for any $r>0$. Recall in \eqref{eq:growthu} that $\alpha_n = n\omega_{n-1}^{1/(n-1)}$. Now \underline{we prove \eqref{eq:SuperMT1*}} under the three conditions ($f_1$), ($f_2$), and ($f_3$). We note by \eqref{eq:growthu} that
\[
f(r) |u(r)|^{\frac n{n-1}} \leq \frac{n}{\alpha_n} (-\log r) f(r)=: g(r)
\]
for any $r> 0$. In view of ($f_3$), there is some $\rho \in (0,1)$ such that
\[
f(r) \leq \gamma \frac{\alpha_n}{n} \frac{\log (1-r)}{\log r}
\] 
for $r\in (\rho,1)$. Then, we have
\begin{align*}
\int_B \exp \big( [\alpha_n + f(|x|)] |u|^{\frac n{n-1}} \big) dx &=\om_{n-1} \int_0^\rho \exp \big( [\alpha_n + f(r)] |u|^{\frac{n}{n-1}} \big) r^{n-1} dr\\
&\quad + \om_{n-1} \int_\rho^1 \exp \big( [\alpha_n + f(r)] |u|^{\frac{n}{n-1}} \big) r^{n-1} dr\\
& = I + II.
\end{align*}
We have by \eqref{eq:growthu} and the estimate for $f$ on $(\rho,1)$ the following estimate for $II$
\begin{align*}
II \leq \om_{n-1} \int_\rho^1 r^{-1} (1-r)^{-\gamma} dr \leq \frac1\ga \frac{(1-\rho)^\ga}{1-\ga}.
\end{align*}
By the condition ($f_2$) on $f$, we see that $g$ is continuous on $(0,1)$ and bounded near $0$. This allows us to set
\[
C_\rho = \sup_{r\in (0,\rho)} g(r),
\]
which is finite. Hence, the term $I$ is easily estimated as follows
\begin{align*}
I & \leq e^{C_\rho} \om_{n-1} \int_0^\rho \exp \big( \alpha_n |u|^{\frac n{n-1}} \big) r^{n-1} dr\\
& \leq e^{C_\rho} \int_B \exp \big( \alpha_n |u|^{\frac n{n-1}} \big) dx
 \leq e^{C_\rho} \MT.
\end{align*}
Putting these estimates together, we arrive at
\[
\int_B \exp \big( [\alpha_n + f(|x|)] |u|^{\frac n{n-1}} \big) dx \leq \frac1\ga \frac{(1-\rho)^\ga}{1-\ga} + e^{C_\rho} \MT.
\]
This proves \eqref{eq:SuperMT1*}. Finally, \underline{we prove \eqref{eq:SuperMT2*}} under the three conditions ($f_1$), ($f'_2$), and ($f_3$). First, we let $\rho = \exp (-\alpha_n/n)$ and by \eqref{eq:growthu} we easily check that $u(r) \leq 1$ for any $r\in (\rho,1)$. This yields
\[
\int_{B\setminus B_\rho} \exp \big( \alpha_n |u|^{\frac n{n-1}+ f(|x|)} \big) dx \leq e^{\alpha_n} \frac{\om_{n-1}}n (1 -\rho^n) \leq e^{\alpha_n} \frac{\om_{n-1}}n.
\]
On $B_\rho$ we have
\begin{align*}
\int_{B_\rho} & \exp \big( \alpha_n |u|^{\frac n{n-1}+ f(|x|)} \big) dx \\
&=\int_{B_\rho} \exp \big( \alpha_n |u|^{\frac n{n-1}}|u|^{f(|x|)} \big) dx\\
&\leq \int_{B_\rho} \exp \Big( \alpha_n |u|^{\frac n{n-1}} \big(-\frac n{\alpha_n} \log |x|\big)^{\frac{n-1}n f(|x|)} \Big) dx\\
&=\int_{B_\rho} \exp \big( \alpha_n |u|^{\frac n{n-1}} \big)  \Big[ \exp\Big( \alpha_n |u|^{\frac n{n-1}}\Big[\big(-\frac n{\alpha_n} \log |x|\big)^{\frac{n-1}n f(|x|)} -1\Big] \Big) -1 \Big] dx\\
& \quad + \int_{B_\rho} \exp \big( \alpha_n |u|^{\frac n{n-1}} \big) dx\\
&\leq \int_{B_\rho} \exp \big( \alpha_n |u|^{\frac n{n-1}} \big)  \Big[ \exp\Big( \alpha_n |u|^{\frac n{n-1}}\Big[\big(-\frac n{\alpha_n} \log |x|\big)^{\frac{n-1}n f(|x|)} -1\Big] \Big) -1 \Big] dx \\
& \quad + \MT .
\end{align*}
Notice that for $r$ near $0$ we have
\begin{align*}
\exp \Big( \alpha_n |u|^{\frac n{n-1}}&\Big[\big(-\frac n{\alpha_n} \log r\big)^{\frac{n-1}n f(r)} -1\Big] \Big) -1 \\
&\leq \exp \Big( -n \log r \Big[\big(-\frac n{\alpha_n} \log r\big)^{\frac{n-1}n \frac{c}{(-\log r)^\ga}} -1\Big] \Big)-1 =:k(r).
\end{align*}
For $r$ near $0$, we have the following expansion
\begin{align*}
\big(-\frac n{\alpha_n} \log r\big)^{\frac{n-1}n \frac{c}{(-\log r)^\ga}} & 
= \exp \Big( \frac{n-1}n \frac{c}{(-\log r)^\ga} \log (-\frac{n}{\alpha_n} \log r) \Big) \\
&\sim 1 + \frac{n-1}n \frac{c}{(-\log r)^\ga} \log (-\frac{n}{\alpha_n} \log r),
\end{align*}
which implies 
\[
-n \Big[\big(-\frac n{\alpha_n} \log r\big)^{\frac{n-1}n \frac{c}{(-\log r)^\ga}} -1\Big] \log r \to 0
\]
as $r \to 0$, thanks to $\gamma > 1$. Hence, when $r$ is near $0$, we get
\[
k(r) \sim \frac{c(n-1)}{(-\log r)^{\ga-1}}\log (-\frac{n}{\alpha_n} \log r)=: h(r).
\]
Since $k$ and $h$ are continuous and strictly positive on $(0,\rho)$. Hence, there is $C'$ such that $$k(r) \leq C' h(r)$$ for any $r \in (0,\rho)$. Hence, together with \eqref{eq:growthu}, we estimate
\begin{align*}
 \int_{B_\rho} \exp \big( \alpha_n |u|^{\frac n{n-1}} \big) & \Big[ \exp\Big( \alpha_n |u|^{\frac n{n-1}}\Big[\big(-\frac n{\alpha_n} \log |x|\big)^{\frac{n-1}n f(|x|)} -1\Big] \Big) -1 \Big] dx \\
&\leq \om_{n-1}C'\int_0^\rho r^{-1}\frac{c(n-1)}{(-\log r)^{\ga-1}}\log (-\frac{n}{\alpha_n} \log r) dr\\
&=\om_{n-1}C' c(n-1)\int_{-\log \rho}^\infty t^{1-\gamma} \log (\frac{n}{\alpha_n}t) dt\\
&< +\infty
\end{align*}
since $\gamma >2$. Finally, we get
\begin{align*}
\int_{B} \exp \big( \alpha_n |u|^{\frac n{n-1}+ f(|x|)} \big) dx \leq & e^{\alpha_n} \frac{\om_{n-1}}n + \MT \\
&+ \om_{n-1}C' c(n-1)\int_{-\log \rho}^\infty t^{1-\gamma} \log \big(\frac{n}{\alpha_n}t \big) dt < +\infty.
\end{align*}
This proves \eqref{eq:SuperMT2*}.
\end{proof}

Note that the function $f(r) =r^\al$ with $\alpha >0$ satisfies all conditions of Theorem \ref{SuperMT*} above; hence Theorem \ref{SuperMT} follows from Theorem \ref{SuperMT*}. We further notice that because of the condition ($f_2$), we cannot apply Theorem \ref{SuperMT*} for the ahove function $f$ with $\alpha = 0$, namely, $f(r) \equiv 1$.

In the last paragraph of this section, we prove Lemma \ref{Sharpness}. Considering the sequence of Moser's functions 
\[
u_j(x) = \om_{n-1}^{-1/n}
\left\{
\begin{aligned}
&  (\log j)^{\frac{n-1}n} & & \mbox{ if } |x| \leq 1/j,\\
& -  \frac{\log |x|} { (\log j)^{1/n}} & & \mbox{ if } 1/j \leq |x| < 1,
\end{aligned}
\right.
\]
with $j \geq 1$. It is easy to check that $u_j \in \Wnr$ and $\|\na u_j\|_n =1$ for any $j$. Now we estimate $\int_B \exp\big((\ga + |x|^\al) |u_j|^{\frac n{n-1}}\big) dx$. For any $\ga > \al_n$, we get
\begin{align*}
\int_B \exp\big((\ga + |x|^\al) |u_j|^{\frac n{n-1}}\big) dx &\geq \int_{B_{1/j}(0)} \exp\big(\ga  |u_j|^{\frac n{n-1}}\big) dx\\
&=\om_{n-1} \int_0^{1/j} \exp\big(\frac{\ga n}{\al_n} \log j\big) r^{n-1} dr\\
&= \frac{\om_{n-1}}n \exp\Big(n\big(\frac{\ga}{\al_n} -1\big) \log j\Big) \to +\infty
\end{align*}
as $j\to +\infty$. To estimate the integral $\int_B \exp\big(\ga |u_j|^{\frac n{n-1} + |x|^\al}\big) dx$, we observe from the definition of $u_j$ that $u_j(x)\geq 1$ for any $x \in B_{1/j}(0)$ if  $j$ is large enough. Consequently, for $j$ large enough, we have
\begin{align*}
\int_B \exp\big(\ga |u_j|^{\frac n{n-1} + |x|^\al}\big) dx &\geq \int_{B_{1/j}(0)} \exp\big(\ga  |u_j|^{\frac n{n-1}+ |x|^\alpha}\big) dx\\
&\geq  \int_{B_{1/j}(0)} \exp\big(\ga  |u_j|^{\frac n{n-1}}\big) dx\\
&= \frac{\om_{n-1}}n \exp\Big(n\big(\frac{\ga}{\al_n} -1\big) \log j\Big) \to +\infty
\end{align*}
as $j\to +\infty$.

\section{The attainability of $\MT^i$ with $i=1,2$: Proof of Theorem \ref{Attain}}

Inspired by \cite{CC1986, FOR02}, a sequence $( u_j )_j\subset \Wnr$ is called a \textit{normalized concentrating sequence}, denoted by NCS for short, if
\begin{itemize}
 \item $\|\nabla u_j\|_{L^n(B)} =1$ for any $j$, 
 \item $u_j \rightharpoonup 0$ weakly in $\Wnr$ as $j \to +\infty$, and 
 \item $\int_{B \setminus B_a} |\nabla u_j|^n dx = o(1)_{j \nearrow +\infty}$ for any $a \in (0,1)$.
\end{itemize} 
Let us define the concentrating level by
\[
J:= \sup\lt\{\limsup_{j\to +\infty} \int_B \exp \big( \alpha_n |u_j|^{\frac n{n-1}} \big) dx\, :\, ( u_j )_j \subset \Wnr,\, ( u_j ) \, \text{\rm is NCS}\rt\}.
\]
It follows from the Moser--Trudinger inequality \eqref{eq:Moser--Trudinger}, see also Theorem \ref{SuperMT}, that $J$ is finite. Furthermore, it is well-known that
\[
J = |B| + |B| \exp \Big( \sum_{i=1}^{n-1} \frac 1i \Big);
\]
see \cite[Theorem 1.4]{FOR02}.

\subsection{Preliminaries}

In the following two lemmas, we estimate $\MT^i (\alpha)$ from below.

\begin{lemma}\label{Strict}
We have $$\MT^1(\al) > \MT$$ for any $\alpha >0$.
\end{lemma}

\begin{proof}
Recall from \cite{CC1986, Lin1996} that the sharp constant $\MT$ is attained by some radial function $u \in \Wnr$ with $\|\nabla u\|_{L^n(B)} =1$. Hence
\[
\MT = \int_B \exp \big( \alpha_n |u|^{\frac n{n-1}} \big) dx < \int_B \exp \big( (\alpha_n+ |x|^\al) |u|^{\frac n{n-1}} \big) dx \leq \MT^1(\al),
\]
completing the proof.
\end{proof}

\begin{lemma}\label{Strict*}
We have $$\MT^2(\al) > J$$ for any $\alpha >0$.
\end{lemma}
\begin{proof}
We consider the test function $u_\varepsilon$ defined by 
\[
u_\varepsilon (x) =
\left\{
 \begin{aligned}
& c + c^{-\frac1{n-1}} \Big[-\frac{n-1}{\alpha_n} \log \Big(1 + \big(\frac{\om_{n-1}}n\big)^{\frac1{n-1}} \big( \frac{|x|}{\varepsilon} \big)^{\frac n{n-1}} \Big) + A\Big] &\mbox{ if } 0\leq |x| \leq R\varepsilon,\\
& c^{-\frac1{n-1}} \big(-\frac n{\alpha_n} \log |x|\big) &\mbox{ if } R\varepsilon \leq |x| \leq 1,
\end{aligned}
\right.
\]
where $R = -\log \ep$ and $c =c_\varepsilon$ and $A = A_\varepsilon$ are chosen in such a way that $u_\varepsilon \in W_0^{1,m}(B)$ and
\[
\|\nabla u_\varepsilon\|_{L^n(B)} =1.
\] 
In order for $u_\varepsilon \in W_0^{1,m}(B)$, we choose $c, A$ in such a way that $u_\varepsilon$ is continuous on $B$, i.e.,
\[
c^{\frac n{n-1}} + A = -\frac{n}{\alpha_n} \log (R\varepsilon) + \frac{n-1}{\alpha_n} \log \Big(1 + \big(\frac{\om_{n-1}}n\big)^{\frac1{n-1}}R^{\frac n{n-1}}\Big),
\]
which yields
\[
c^{\frac n{n-1}} + A = -\frac n{\alpha_n} \log \varepsilon + \frac1{\alpha_n} \log \frac{\om_{n-1}}n + O(R^{-\frac n{n-1}}).
\]
By a simple computation, we have
\[
\int_B |\nabla u_\varepsilon|^n dx = c^{-\frac n{n-1}} \Big( - \frac n{\alpha_n} \log \varepsilon + \frac1{\alpha_n} \log \frac{\om_{n-1}}n -\frac{n-1}{\alpha_n} \sum_{i=1}^{n-1} \frac1i + O(R^{-\frac n{n-1}})\Big).
\]
Hence, the requirement $\|\nabla u_\varepsilon\|_{L^n(B)} =1$ implies 
\begin{equation}\label{eq:c}
c^{\frac n{n-1}} = - \frac n{\alpha_n} \log \varepsilon + \frac1{\alpha_n} \log \frac{\om_{n-1}}n -\frac{n-1}{\alpha_n} \sum_{i=1}^{n-1} \frac1i + O(R^{-\frac n{n-1}}).
\end{equation}
Now we estimate $\int_B \exp \big( \alpha_n |u_\varepsilon|^{\frac n{n-1} + |x|^\alpha} \big) dx$. Notice that $c \to +\infty$ as $\varepsilon \searrow 0$. Let
\[
a = \big(\frac n{2(n-1/2)^2}\big)^{1/\alpha} < 1.
\]
From now on, we let $\varepsilon \ll 1$ be such that $R\varepsilon < a$. Clearly, for $|x| > R\varepsilon$, by the elementary inequality $e^t \geq 1 + t^n/n!$ with $t>0$, we obtain
\begin{align*}
\int_{B\setminus B_{R\varepsilon}} & \exp \big( \alpha_n |u_\varepsilon|^{\frac n{n-1} + |x|^\alpha} \big) dx \\
&\geq \int_{B\setminus B_{R\varepsilon}} \Big(1 + \frac{\alpha_n^{n-1}}{(n-1)!} |u_\varepsilon|^{n + {(n-1)|x|^\alpha}} \Big) dx\\
&=|B| + O((R\varepsilon)^n) + \frac{c^{-\frac n{n-1}} \alpha_n^{n-1}}{(n-1)!} \int_{B\setminus B_{R\varepsilon}}\big(-\frac n{\alpha_n} \log |x|\big)^{n + (n-1) |x|^\alpha} c^{-|x|^\alpha} dx\\
&\geq |B| + O(R^{-\frac n{n-1}}) + \frac{c^{-\frac n{n-1}} \alpha_n^{n-1}}{(n-1)!}\int_{B_a\setminus B_{R\varepsilon}}\big(-\frac n{\alpha_n} \log |x|\big)^{n + (n-1) |x|^\alpha} c^{-|x|^\alpha} dx\\
&\geq |B| + O(R^{-\frac n{n-1}}) + \frac{c^{-\frac n{n-1}- a^\alpha} \alpha_n^{n-1}}{(n-1)!}\int_{B_a\setminus B_{R\varepsilon}}\big(-\frac n{\alpha_n} \log |x|\big)^{n + (n-1) |x|^\alpha} dx\\
&=|B| + O(R^{-\frac n{n-1}}) + \frac{c^{-\frac n{n-1}- a^\alpha} \alpha_n^{n-1}}{(n-1)!}\Big(\int_{B_a}\big(-\frac n{\alpha_n} \log |x|\big)^{n + (n-1) |x|^\alpha} dx + o(1)\Big).
\end{align*}
On $B_{R\varepsilon}$, due to the monotonicity of $u_\varepsilon$, we have
\begin{align*}
u_\varepsilon(x) \geq u_\varepsilon(R\varepsilon) = c^{-\frac1{n-1}} \frac n{\alpha_n}\big[-\log \varepsilon - \log (-\log \varepsilon)\big]\geq 1
\end{align*}
provided $\varepsilon$ is small enough. Using the inequality $(1+ t)^{n/(n-1)} \geq 1 + \frac n{n-1} t$ for $t> -1$ and the estimates for $c$ and $A$, we get
\begin{align*}
\al_n |u_\varepsilon(x)|^{\frac n{n-1}}&\geq \al_n c^{\frac n{n-1}} -n \log \Big(1 + \big(\frac{\om_{n-1}}n\big)^{\frac1{n-1}} \big( \frac{|x|}{\varepsilon} \big)^{\frac n{n-1}} \Big) + \frac {n\al_n}{n-1} A\\
&=\frac{n \al_n}{n-1}(c^{\frac n{n-1}} + A) -\frac{\al_n}{n-1} c^{\frac n{n-1}} -n \log \Big(1 + \big(\frac{\om_{n-1}}n\big)^{\frac1{n-1}} \big( \frac{|x|}{\varepsilon} \big)^{\frac n{n-1}} \Big)\\
&= -n \log \varepsilon  + \log \frac{\om_{n-1}}n + \sum_{i=1}^{n-1} \frac 1i + O(R^{-\frac{n}{n-1}}) -n \log \Big(1 + \big(\frac{\om_{n-1}}n\big)^{\frac1{n-1}} \big( \frac{|x|}{\varepsilon} \big)^{\frac n{n-1}} \Big). 
\end{align*}
Hence, 
\begin{align}\label{eq:J}
\int_{B_{R\varepsilon}} \exp \big( \alpha_n |u_\varepsilon|^{\frac n{n-1} + |x|^\alpha} \big) dx &\geq \int_{B_{R\varepsilon}} \exp \big( \alpha_n |u_\varepsilon|^{\frac n{n-1}} \big) dx\notag\\
&\geq  e^{\sum_{i=1}^{n-1} \frac1i + O(R^{-\frac n{n-1}})} \varepsilon^{-n} \frac{\om_{n-1}}n \int_{B_{R\varepsilon}} \Big(1 + \big(\frac{\om_{n-1}}n\big)^{\frac1{n-1}} \big( \frac{|x|}{\varepsilon} \big)^{\frac n{n-1}} \Big)^{-n} dx\notag\\
&= |B| \exp \big( \sum_{i=1}^{n-1} \frac 1i \big) + O(R^{-\frac n{n-1}}).
\end{align}
Finally, we get
\begin{align*}
\int_{B} \exp \big( \alpha_n & |u_\varepsilon|^{\frac n{n-1} + |x|^\alpha} \big) dx \\
&\geq |B| + |B| \exp \big( \sum_{i=1}^{n-1} \frac 1i \big) + O(R^{-\frac n{n-1}})\\
&\quad + \frac{c^{-\frac n{n-1}- a^\alpha} \alpha_n^{n-1}}{(n-1)!}\Big(\int_{B_a}\Big(-\frac n{\alpha_n} \log |x|\Big)^{n + (n-1) |x|^\alpha} dx + o(1)\Big).
\end{align*}
By the choice of $a$, we have
\[
\frac{n}{n-1} + a^\alpha = \frac{n}{n-1} \Big(1 + \frac1{2(n-1/2)}\Big) = \frac{n}{n-1} \frac{n}{n-1/2} < \Big(\frac{n}{n-1}\Big)^2
\]
for any $n \geq 2$. Observe that $c^{\frac n{n-1}} \equiv R$ as $\varepsilon \to 0$. Hence, for $\varepsilon >0$ small enough, the term $O(R^{-\frac n{n-1}})$ is absorbed into the integral $\int_{B_a\setminus B_{R\varepsilon}}$, which yields
\[
\MT^2(\al) \geq \int_{B} \exp \big( \alpha_n |u_\varepsilon|^{\frac n{n-1} + |x|^\alpha} \big) dx > |B| + |B| \exp \Big( \sum_{i=1}^{n-1} \frac 1i \Big) =J.
\]
The proof follows.
\end{proof}

\begin{lemma}[concentration-compactness principle of Lions type] \label{CC}
Let $( u_j )_j$ be a sequence in $\Wnr$ with $\|\nabla u_j\|_{L^n(B)} =1$ for each $j$ and $u_j \rightharpoonup u$ weakly in $\Wnr$. Then for any $$p < (1 - \|\nabla u\|_{L^n(B)})^{-\frac1{n-1}}$$ there hold
\begin{equation}\label{eq:CC1}
\limsup_{j\to +\infty} \int_B \exp \big( p (\alpha_n+ |x|^\al) |u_j|^{\frac n{n-1}} \big) dx < +\infty
\end{equation}
and 
\begin{equation}\label{eq:CC2}
\limsup_{j\to +\infty} \int_B \exp \big( p \alpha_n |u_j|^{\frac n{n-1} + |x|^\alpha} \big) dx < +\infty.
\end{equation}
\end{lemma}
\begin{proof}
Taking $a >0$ small enough such that
\[
q = p \big(1+ \frac{a^\alpha}{\alpha_n} \big) < (1 - \|\nabla u\|_{L^n(B)})^{-\frac1{n-1}}.
\] 
From \eqref{eq:growthu}, we have
\[
|u_j(r)|\leq \big(-\frac{n}{\alpha_n} \log a \big)^{\frac{n-1}n}
\]
for a.e. $r \in (a,1)$. Consequence, 
\begin{equation}\label{eq:truncation1}
\int_{B\setminus B_a} \exp \big( p (\alpha_n+ |x|^\al) |u_j|^{\frac n{n-1}} \big) dx \leq a^{-pn (1+ \frac1{\alpha_n})} |B|.
\end{equation}
On the other hand, we have
\begin{align*}
\int_{B_a} \exp \big( p (\alpha_n+ |x|^\al) |u_j|^{\frac n{n-1}} \big) dx & \leq \int_{B_a} \exp \big( p (\alpha_n+ a^\al) |u_j|^{\frac n{n-1}} \big)dx \\
&\leq \int_{B} \exp \big( q \alpha_n |u_j|^{\frac n{n-1}} \big) dx.
\end{align*}
The classical concentration-compactness principle in \cite[Section I.7]{Lions1985} implies
\[
\limsup_{j\to +\infty} \int_{B_a} \exp \big( p (\alpha_n+ |x|^\al) |u_j|^{\frac n{n-1}} \big) dx < +\infty,
\]
which together with \eqref{eq:truncation1} proves \eqref{eq:CC1}. A similar argument also confirms \eqref{eq:CC2}.
\end{proof}

\subsection{Proof of Theorem \ref{Attain}: The case of $\MT^1(\al)$}

Let $( u_j )_j$ be a maximizing sequence for $\MT^1(\al)$. We can assume that 
\begin{itemize}
 \item $u_j\rightharpoonup u$ weakly in $\Wnr$ and
 \item $u_j \to u$ a.e. in $B$. 
\end{itemize} 
Suppose that $u\equiv 0$. By \eqref{eq:growthu} and Lebesgue's dominated convergence theorem, we can write
\begin{align*}
\int_{B\setminus B_a} \exp \big( (\alpha_n+ |x|^\al) |u_j|^{\frac n{n-1}} \big) dx
& =|B\setminus B_a| + o(1)_{j \searrow 0} \\
&= \int_{B\setminus B_a} \exp \big( \alpha_n|u_j|^{\frac n{n-1}} \big) dx + o(1)_{j \searrow 0}
\end{align*}
for arbitrary but fixed $a>0$. On the ball $B_a$ with $a < e^{-1/\alpha}$, still by \eqref{eq:growthu} and the monotonicity of the function $-t^\alpha \log t$ on $(0, a)$, we have
\begin{align*}
\int_{B_a} \exp \big( (\alpha_n+ |x|^\al) |u_j|^{\frac n{n-1}} \big) dx& \leq \int_{B_a} e^{-\frac n{\alpha_n} |x|^\alpha (-\log |x|)} \exp \big( \alpha_n |u_j|^{\frac n{n-1}} \big) dx \\
&\leq e^{-\frac n{\alpha_n} a^\alpha (-\log a)} \int_{B_a} \exp \big( \alpha_n |u_j|^{\frac n{n-1}} \big) dx.
\end{align*}
Hence
\begin{align*}
\int_{B} \exp \big( (\alpha_n+ |x|^\al) |u_j|^{\frac n{n-1}} \big) dx &= \Big( \int_{B_a} + \int_{B \setminus B_a} \Big) \exp \big( (\alpha_n+ |x|^\al) |u_j|^{\frac n{n-1}} \big) dx \\
& \leq e^{\frac n{\alpha_n} a^\alpha \log a} \int_B \exp \big( \alpha_n |u_j|^{\frac n{n-1}} \big) dx + o(1)_{j \searrow 0}\\
&\quad + (1-e^{\frac n{\alpha_n} a^\alpha \log a})\int_{B\setminus B_a} \exp \big( \alpha_n |u_j|^{\frac n{n-1}} \big) dx \\
&\leq e^{\frac n{\alpha_n} a^\alpha \log a} \MT + o(1)_{j \searrow 0} \\
& \quad + (1-e^{\frac n{\alpha_n} a^\alpha \log a})\int_{B\setminus B_a} \exp \big( \alpha_n |u_j|^{\frac n{n-1}} \big) dx.
\end{align*}
Letting $j \nearrow +\infty$ and $a \searrow 0$ we eventually get
 $$\MT^1(\al) \leq \MT,$$ 
 which is impossible by Lemma \ref{Strict}. Hence, $u\not\equiv 0$. This allows us to conclude that $\|\nabla u\|_{L^n(B)} > 0$. Hence, by Lemma \ref{CC}, the function
 \[
 \exp \big( (\alpha_n+ |x|^\al) |u_j|^{\frac n{n-1}} \big) \in L_p(B)
 \] 
 for some $p >1$. From this and by H\"older's inequality we can easily show that the sequence $\exp \big( (\alpha_n+ |x|^\al) |u_j|^{\frac n{n-1}} \big)$ is uniformly integrable. Now one can make use of the Vitali convergence theorem to realize that
\[
\MT^1(\al) = \lim_{j\to +\infty} \int_B \exp \big( (\alpha_n+ |x|^\al) |u_j|^{\frac n{n-1}} \big) dx =\int_B \exp \big( (\alpha_n+ |x|^\al) |u|^{\frac n{n-1}} \big) dx.
\]
A simple argument shows that $\|\nabla u\|_{L^n(B)} \leq 1$ and $u$ is indeed a maximizer for $\MT^1(\al)$.

\subsection{Proof of Theorem \ref{Attain}: The case of $\MT^2(\al)$}

Let $( u_j )_j$ be a maximizing for $\MT^2(\al)$. As before, we can assume that
\begin{itemize}
 \item $u_j\rightharpoonup u$ weakly in $\Wnr$ and
 \item $u_j \to u$ a.e. in $B$. 
\end{itemize} 
As in the case of $\MT^1(\al)$, we need to rule out the possibility of $u \equiv 0$. By way of contradiction, suppose that $u\equiv 0$. As before, by \eqref{eq:growthu} and Lebesgue's dominated convergence theorem, we can also deduce that
\begin{equation}\label{eq:aa1}
\begin{aligned}
\int_{B\setminus B_a} \exp \big( \alpha_n|u_j|^{\frac n{n-1} + |x|^\alpha }\big) dx &=|B\setminus B_a| + o(1)_{j \searrow 0} \\
&= \int_{B\setminus B_a} \exp \big( \alpha_n|u_j|^{\frac n{n-1}} \big) dx + o(1)_{j \searrow 0}
\end{aligned}
\end{equation}
for arbitrary but fixed $a>0$. There are two possible cases: either $( u_j )$ is a NCS or $( u_j )$ is not a NCS. Suppose that the maximizing sequence $( u_j )$ is NCS. 
 Hence, by the first and third conditions in the definition of a NCS, in the present scenario, we must have
\[
\lim_{j\to +\infty} \int_{B_{a}}|\nabla u_j|^n dx = 1
\]
for any $a \in (0,1)$. Our aim is to estimate the integral $\int_{B} \exp ( \alpha_n |u_j|^{\frac n{n-1} + |x|^\alpha} ) dx$. Again by \eqref{eq:growthu} we have 
\begin{align*}
\int_{B_a} \exp &\big( \alpha_n |u_j|^{\frac n{n-1} + |x|^\alpha} \big) dx \\
 &\leq \int_{B_a} \exp \big(\alpha_n |u_j|^{\frac n{n-1}}(-\frac n{\alpha_n} \log |x|)^{\frac{n-1}n |x|^\alpha} \big) dx \\
&= \int_{B_a} \Big(\exp \big( \alpha_n |u_j|^{\frac n{n-1}}\big[ (-\frac n{\alpha_n} \log |x|)^{\frac{n-1}n |x|^\alpha} -1\big] \big) -1\Big) \exp \big( \alpha_n |u_j|^{\frac n{n-1}} \big) dx\\
&\quad + \int_{B_a} \exp \big(\alpha_n |u_j|^{\frac n{n-1}} \big) dx.
\end{align*}
Combining the previous estimate with \eqref{eq:aa1} we get
\begin{align*}
\int_{B} \exp &\big( \alpha_n |u_j|^{\frac n{n-1} + |x|^\alpha} \big) dx \\
&\leq\int_{B_a} \Big( \exp \Big( \alpha_n |u_j|^{\frac n{n-1}}\Big( (-\frac n{\alpha_n} \log |x|)^{\frac{n-1}n |x|^\alpha} -1\Big) \Big) -1 \Big) \exp \big( \alpha_n |u_j|^{\frac n{n-1}} \big) dx \\
&\quad + \int_{B} \exp \big( \alpha_n |u_j|^{\frac n{n-1}} \big) dx + o(1)_{j \searrow 0}.
\end{align*}
In the rest of our argument, we mainly show that the first integral on the right hand side of the preceding inequality is negligible. Since $a$ is arbitrary, we may fix $$a_0 = \exp( - \alpha_n/ n)$$ and consider $r \leq a \leq a_0$. Clearly, $-( n/ \alpha_n) \log r \geq 1$ for any $r \leq a \leq a_0$. Therefore, we can estimate that integral as follows
\begin{align*}
\int_{B_a} \Big( \exp &\Big( \alpha_n |u_j|^{\frac n{n-1}}\Big( (-\frac n{\alpha_n} \log |x|)^{\frac{n-1}n |x|^\alpha} -1\Big) \Big) -1 \Big) \exp \big( \alpha_n |u_j|^{\frac n{n-1}} \big) dx\\
&\leq \omega_{n-1}\int_{0}^a \Big(    \exp \Big( -n \log r \big( (-\frac n{\alpha_n} \log r)^{\frac{n-1}n r^\alpha} -1\big) \Big) -1   \Big) r^{-1} dr.
\end{align*}
Note that
\[
\exp \Big( -n \log r \big( (-\frac n{\alpha_n} \log r)^{\frac{n-1}n r^\alpha} -1\big) \Big) -1 
\sim -(n-1) r^\alpha \log r \, \log \big(-\frac n{\alpha_n} \log r \big)
\]
as $r \searrow 0$. Hence, thanks to $\alpha>0$, we know that
\[
\Big(    \exp \Big( -n \log r \big( (-\frac n{\alpha_n} \log r)^{\frac{n-1}n r^\alpha} -1\big) \Big) -1   \Big) r^{-1} \in L^1(0,a_0),
\] 
which yields
\[
\lim_{a \searrow 0} \int_{0}^a \Big(\exp \Big( -n \log r \Big( \big(-\frac n{\alpha_n} \log r \big)^{\frac{n-1}n r^\alpha} -1\Big) \Big) -1 \Big) r^{-1} dr =0.
\]
By letting $j\nearrow +\infty$ and then $a\searrow 0$, we obtain
\[
\MT^2(\al) \leq \limsup_{j\to +\infty} \int_B \exp \big( \alpha_n |u_j|^{\frac n{n-1}} \big) dx \leq J,
\]
which is impossible by Lemma \ref{Strict*}. Hence, the maximizing sequence $( u_j )_j$ is not a NCS. Consequently, there is a subsequence, still denoted by $(u_j)_j$, and $0< a, \de < 1$ such that
\[
\lim_{j\to +\infty} \int_{B_{a}}|\nabla u_j|^n dx \leq \de < 1.
\]
Let $\varphi \in \Cr$ be a cut-off function such that $0 \leqslant \varphi \leqslant 1$ and
\[
\varphi (x) =
\begin{cases}
0 & \text{ if } |x| \geq 1,\\
1 & \text{ if } |x| \leq 1/2.
\end{cases}
\]
We define
\[
\varphi_a = \varphi(\cdot/a)
\] 
for each $a >0$. For any $a < r_0$ we have $\nabla (\varphi_a u_j) = \vphi_a \nabla u_j + u_j \nabla \varphi_a$. Since $u_j\rightharpoonup 0$ weakly in $\Wnr$, we know that $u_j\nabla \varphi_a \to 0$ in $L_n(B)$. Hence
\[
\limsup_{j\to +\infty} \int_B |\na(\vphi_a u_j)|^n dx \leq \limsup_{j\to +\infty} \int_{B_a} |\nabla u_j|^n dx \leq \de.
\]
Because $\int_{B_a}|\nabla u_j|^n dx \leq \de < 1$ for all $j$, by Theorem \ref{SuperMT}, there exists $p >1$ such that
\[
\exp \big(\alpha_n |\varphi_a u_j|^{\frac{n}{n-1} +|x|^\alpha}\big) \in L_p(B).
\] 
In particular, because $\varphi_a \equiv 1$ in $B_{a/2}$, we conclude that
\[
\exp \big(\alpha_n |u_j|^{\frac n{n-1} + |x|^\alpha} \big) \in L_p(B_{a/2}).
\] 
Keep in mind that we are assuming $u\equiv 0$; hence, $\alpha_n |u_j|^{\frac n{n-1} + |x|^\alpha} \to 0$ a.e., which together with the Vitali convergence theorem implies that
\[
\lim_{j\to +\infty} \int_{B_{a/2}} \exp \big(\alpha_n |u_j|^{\frac n{n-1} + |x|^\alpha} \big) dx = |B_{a/2}|.
\]
Consequently, combining the preceding limits and \eqref{eq:aa1} with $a$ replaced by $a/2$ gives
\[
 \lim_{j\to +\infty} \int_{B} \exp \big(\alpha_n |u_j|^{\frac n{n-1} + |x|^\alpha} \big) dx = |B|,
\]
which is impossible by means of Lemma \ref{Strict*}. Hence, we must have $u\not\equiv 0$. By Lemma \ref{CC}, the function
\[
\exp \big( \alpha_n |u_j|^{\frac n{n-1} + |x|^\alpha} \big) \in L_p(B)
\]
for some $p >1$. Hence, still relying on the Vitali convergence theorem as in the case of $\MT^1 (\alpha)$, we can also deduce that
\[
\MT^2(\al) = \lim_{j\to +\infty} \int_B \exp \big(\alpha_n |u_j|^{\frac n{n-1} + |x|^\alpha} \big) dx =\int_B \exp \big(\alpha_n|u|^{\frac n{n-1} + |x|^\alpha} \big) dx.
\]
Now a simple argument shows that $\|\nabla u\|_{L^n(B)} \leq 1$; hence, $u$ is a maximizer for $\MT^2(\al)$.


\section{Application: Proof of Theorem \ref{thmApplication}}

This section is devoted to an existence result for problem \eqref{eq:Super}, namely
\[
\left\{
\begin{aligned}
-\Delta_n u &= f(x,u) && \mbox{ in } B,\\
u&> 0 && \mbox{ in } B,\\
 u &= 0 && \mbox{ on } \pa B,
\end{aligned}
\right.
\]
under the conditions $(F_1)$--$(F_5)$. Our aim is to look for a solution $u$ to \eqref{eq:Super} in the space $\Wnr$. 


\subsection{The variational formulation}

Define
\begin{equation*}\label{FunctionalI}
I(u) = \frac 1n \int_B |\nabla u|^n dx - \int_B F(x,u_+)dx
\end{equation*}
with $u \in \Wnr$, where $F(x,t)$ is already given in Introduction. Standard arguments show that $I$ is well-defined and of class $C^1(\Wnr, \R)$, and for any $v \in \Wnr$ there holds
\[
I'(u) v = \int_B |\na u|^{n-2} \na u \cdot \na v dx - \int_B f(x,u_+) v dx.
\]
Under the hypothesis that $f$ is continuous, we see from \eqref{CG>} that given $\beta > \alpha_0$, there exists some $C_\beta > 0$ such that
\begin{equation}\label{BoundFAbove}
|f(x,t)| \leq C_\beta \exp \big(\beta |t|^{\frac{n}{n-1} + |x|^\alpha } \big)
\end{equation}
for all $(x,t) \in B \times \R$. The condition $(F_2)$ implies the following well-known Ambrosetti--Rabinowitz condition:
\begin{enumerate}
\item[($F'_1$)] There exist $R_0 >0$, $\theta > n$ such that
\[
\theta F(x,t) \leq t f(x,t)
\]
for all $t \geq R_0$ and any $x \in B$. 
\end{enumerate}
It is worth noticing that there have been many works weakening the Ambrosett--Rabinowitz condition; for e.g. see \cite{LamLu2014}. However, to keep our work in a reasonable length and simply to demonstrate a simple application of our supercritical inequalities, we do not treat any possible replacement of $(F'_1)$ in the present paper. Indeed, we can fix any $\theta > n$ and choose $R_0 = \max\{R, \theta M\}$ with $R$ and $M$ from $(F_2)$. Now we rewrite $(F_2)$ as $(f(x,t)/F(x,t)) \geq 1/M>0$ for any $t \geq R$ to get
\[
\Big( \frac{F(x,t)}{e^{t/M}} \Big)' \geq 0
\]
for any $t \geq R$. Hence, the condition $(F_2)$ also yields the following condition:
\begin{enumerate}
\item[($F'_2$)] There exists $C>0$ such that
\[
F(x,t) \geq Ce^{t/M}
\]
for any $t \geq R$ and any $x \in B$.
\end{enumerate}
Finally, it follows immediately from $(F_3)$ and the definition of $F$ that
\begin{enumerate}
\item[($F'_3$)] $F(x,t) \geq 0$ for any $(x,t) \in B \times [0,+\infty)$.
\end{enumerate}

\begin{lemma}\label{lem-I-AtInfinity}
Assume $(F'_2)$. Then 
\[
\lim_{t \to +\infty} I(tu) = -\infty
\]
for any $u \in \Wnr \setminus \{ 0 \}$ with $u \geq 0$.
\end{lemma}

\begin{proof}
Since $u\geq 0$ and $u\not\equiv 0$, there exist $a >0$ such that $|\{u\geq a\}| >0$. For $t> R/a$ we have from ($F_2'$)
\[
\int_B F(x,(t u)_+) dx \geq \int_{\{u\geq a\}} F(x,tu) dx \geq \int_{\{u\geq a\}} Ce^{tu/M} dx \geq C e^{ta/M} |\{u\geq a\}|.
\]
Using this we have
\[
I(tu) \leq \frac {t^n}n \int_B |\nabla u|^n dx - C e^{ta/M} |\{u\geq a\}|
\]
for any $t > R/a$. Hence the proof follows.
\end{proof}

From now on, let $\beta>\alpha_n$ be determined later.

\begin{lemma}\label{lem-I-SomeWhere}
Assume $(F_4)$. Then there exist $\delta, \rho > 0$ such that
\[
I(u) \geq \delta
\]
if $\|\na u\|_n = \rho$.
\end{lemma}

\begin{proof}
Given $q>n$, in view of $(F_2)$, $(F_4)$, and \eqref{BoundFAbove} we can choose some $\lambda < \lambda_1 (B)$ and some $C(\beta, \lambda,q)>0$ such that
\[
F(x,t) \leq \frac \lambda n |t|^n + C(\beta, \lambda,q) \exp(\beta |t|^{\frac{n}{n-1} + |x|^\alpha}) |t|^q
\]
for any $(x,t) \in B \times [0,+\infty)$. By H\"older's inequality, with $s>1$ and $1/s+1/r=1$, we get
\begin{align*}
\int_B \exp (\beta& |u|^{\frac{n}{n-1} + |x|^\alpha}) |u|^q dx \\
& \leq \Big( \int_B \exp (\beta r \|\na u\|_n^{\frac{n}{n-1}+ |x|^\alpha}     \big|\frac u{\|\na u\|_n} \big|^{\frac{n}{n-1} + |x|^\alpha}) dx \Big)^{1/r} \Big( \int_B |u|^{qs} dx \Big)^{1/s}.
\end{align*}
Hence, on one hand if we let $\|\na u\|_n \leq \sigma < 1$ with $\beta r \sigma^{\frac{n}{n-1} } < \alpha_n$, then our supercritical Moser--Trudinger inequality \eqref{eq:SuperMT2} tells us that
\begin{align*}
\int_B \exp (\beta& |u|^{\frac{n}{n-1} + |x|^\alpha}) |u|^q dx \leq (\MT^2 (\alpha))^{1/r} \Big( \int_B |u|^{qs} dx \Big)^{1/s}.
\end{align*}
On the other hand, because the embedding $\Wnr \hookrightarrow L_{qs}(B)$ is always continuous, there holds
\[
\Big( \int_B |u|^{qs} dx \Big)^{1/s} \leq C \|\na u\|_n^q
\]
for some dimensional constant $C>0$. Putting these facts together, we are able to estimate $I$ from below as follows
\[
I(u) \geq \frac 1n \Big( 1 - \frac \lambda{\lambda_1(B)} \Big) \|\na u\|_n^n - C \|\na u\|_n^q
\]
for some $C>0$. From this, the proof follows because $\lambda < \lambda_1(B)$ and $q > n$.
\end{proof}

\subsection{The Palais--Smale condition}

Now we prove that under the condition $(F_5)$, the functional $I$ satisfies the well-known Palais--Smale compactness condition. As remarked earlier, our condition $(F_5)$ is weaker than a similar condition in \cite{doO1996}, we cannot use the test function in \cite{doO1996}.

\begin{lemma}\label{lem-I-PS}
Assume $(F_4)$ and $(F_5)$. Then there exists some $j$ such that
\[
\max_{t \geq 0} I(tM_j) < \frac 1n \big( \frac{\alpha_n}{\alpha_0} \big)^{n-1},
\]
where $M_j = u_{1/j}$ with $u_{1/j}$ given in the proof of Lemma \ref{Strict*}, namely
\[
M_j (x) = 
\left\{
\begin{aligned}
&c_{1/j} + c_{1/j}^{-\frac1{n-1}} \Big[A_{1/j} -\frac{n-1}{\alpha_n} \log \Big(1 + \big(\frac{\om_{n-1}}n ( j |x| )^n\big)^{\frac1{n-1}}   \Big)  \Big]
&& \text{ if } |x| \leq \frac{\log j}j,\\
&c_{1/j}^{-\frac1{n-1}} \big(-\frac n{\alpha_n} \log |x|\big)  
 && \text{ if } \frac{\log j}j \leq |x|  .
\end{aligned}
\right.
\]
where $c_{1/j}$ and $A_{1/j}$ are chosen in such a way that $M_j \in W_0^{1,m}(B)$ and
\[
\|\nabla M_j\|_{L^n(B)} =1.
\] 
\end{lemma}

\begin{proof}
By way of contradiction, suppose that
\[
\max_{t \geq 0} I(tM_j) \geq \frac 1n \big( \frac{\alpha_n}{\alpha_0} \big)^{n-1}
\]
for all $j$. In view of Lemma \ref{lem-I-AtInfinity}, there is some $t_j>0$ such that 
\[
I(t_jM_j) = \max_{t \geq 0} I(tM_j) .
\]
Since $F(x,tM_j) \geq 0$ in $B$ and $\|\na M_j\|_n =1$, we deduce that
\[
t_j \geq \big( \frac{\alpha_n}{\alpha_0} \big)^\frac{n-1}n.
\]
Because $\frac d{dt} I(tM_j) \big|_{t = t_j} = 0$, we know that
\[
t_j^n = \int_B t_j M_j f(x, t_jM_j) dx.
\]
Given $\varepsilon \in (0 , \beta_0) $, by using $(F_5)$, there is some $R_\varepsilon >0$ such that
\[
t f(x,t) \geq (\beta_0 - \varepsilon) \exp(\alpha_0 |t|^{\frac n{n-1} })
\]
uniformly on $x$ and for any $t \geq R_\varepsilon$. Thus, by starting from large $j$ such that $t_jM_j \geq R_\varepsilon$ everywhere in the ball $B_{\log j/j}(0)$ and because $M_j$ is strictly decreasing, we can estimate
\begin{align*}
t_j^n & \geq (\beta_0 - \varepsilon) \int_{B_{\log j /j}(0)} \exp(\alpha_0 |t_j M_j|^{\frac n{n-1} }) dx\\
&\geq (\beta_0 - \varepsilon)\int_{B_{\log j /j}(0)} \exp \big(\alpha_0 |t_j M_j \big( \frac{\log j}j  \big) |^{\frac n{n-1} } \big) dx\\
&= (\beta_0 - \varepsilon) |B_{\log j /j}(0)| \exp \big(\alpha_0 |t_j M_j \big( \frac{\log j}j  \big) |^{\frac n{n-1} } \big)  \\
&= (\beta_0 - \varepsilon) \frac{\omega_{n-1}}{n} \exp \Big(-n \big[\log \frac{\log j}j \big] \big[\frac {\alpha_0}{\alpha_n} t_j^{\frac n{n-1} }\big(\frac{-\frac n{\al_n} \log \frac{\log j}j}{c_{1/j}^{n/(n-1)}}\big)^{\frac1{n-1}} - 1 \big] \Big).
\end{align*}
Notice that 
\[
\lim_{j\to \infty} \frac{-\frac n{\al_n} \log \frac{\log j}j}{c_{1/j}^{ n/(n-1)}} =1
\]
by \eqref{eq:c}. Hence $(t_j)_j$ is bounded. Furthermore, due to the presence of the term $ -n \log \frac{\log j}j$ in the preceding inequality, it also implies that
\[
\lim_{j \to +\infty} t_j = \big( \frac{\alpha_n}{\alpha_0} \big)^\frac{n-1}n.
\]
Denote
\[
A_j = \{ x \in B : t_j M_j(x) \geq R_\varepsilon\}, \quad B_j = B \setminus A_j.
\]
Clearly,
\begin{align*}
t_j^n & \geq (\beta_0 - \varepsilon) \int_B \exp(\alpha_0 |t_j M_j|^{\frac n{n-1}  }) + \int_{B_j} t_j M_j f(x, t_jM_j) dx\\
&\quad -(\beta_0 - \varepsilon) \int_{B_j} \exp(\alpha_0 |t_j M_j|^{\frac n{n-1}  }).
\end{align*}
Notice that $M_j \to 0$ a.e. in $B$ and $\chi_{B_j} \to 1$ a.e. in $B$. Therefore, we can apply the Vitali convergence theorem to conclude that
\[
\int_{B_j} t_j M_j f(x, t_jM_j) dx \to 0
\]
and
\[
\int_{B_j} \exp(\alpha_0 |t_j M_j|^{\frac n{n-1} }) \to |B|
\]
as $j \to +\infty$. We also note that
\begin{align*}
\int_B \exp(\alpha_0 |t_j M_j|^{\frac n{n-1} }) &\geq \int_B \exp(\alpha_n |M_j|^{\frac n{n-1} }) dx\\
&=\int_{B\setminus B_{\log j/j}(0)} \exp(\alpha_n |M_j|^{\frac n{n-1} }) dx + \int_{B_{\log j/i}(0)} \exp(\alpha_n |M_j|^{\frac n{n-1} })dx.
\end{align*}
Notice that 
\[
\lim_{j\to \infty} \int_{B\setminus B_{\log j/j}(0)} \exp(\alpha_n |M_j|^{\frac n{n-1} }) dx = |B|.
\]
Hence, letting $j\to \infty$ and using \eqref{eq:J},we get
\[
\lim_{j\to \infty} \int_B \exp(\alpha_n |M_j|^{\frac n{n-1} }) dx \geq |B| + |B| e^{1+ \cdots + \frac1{n-1}}.
\]
Passing to the limit as $j \to +\infty$, we obtain
\[
 \big( \frac{\alpha_n}{\alpha_0} \big)^{n-1} \geq (\beta_0 - \varepsilon) \frac{\omega_{n-1}}n e^{1+ \cdots + \frac1{n-1}},
\]
for any $\epsilon \in (0, \beta_0)$, which implies
\[
\beta_0 \leq \frac{n^n}{\alpha_0^{n-1} e^{1+ \cdots + \frac1{n-1}}}.
\]
This is a contradiction.
\end{proof}

Next we establish an important convergence result.

\begin{lemma}\label{lem-I-Convergence}
Let $(u_j)_j \subset \Wnr$ be a Palais--Smale sequence. Then, up to a subsequence, there is some $u \in \Wnr$ such that
\[
f(x,(u_j)_+) \to f(x,u_+)
\]
in $L_1 (B)$ and 
\[
|\nabla u_j|^{n-2}\nabla u_j \rightharpoonup |\na u|^{n-2} \na u
\]
weakly in $(L_{n/(n-1)}(B))^n$.
\end{lemma}

\begin{proof}
Let $(u_j)_j \subset \Wnr$ be a Palais--Smale sequence, namely,
\begin{equation}\label{eq:PS-1}
\frac 1n \int_B |\nabla u_j|^n dx - \int_B F(x,(u_j)_+) dx \to c
\end{equation}
and
\begin{equation}\label{eq:PS-2}
\Big| \int_B |\nabla u_j|^{n-2}\nabla u_j \cdot \nabla \phi dx -  \int_B f(x,(u_j)_+)\phi dx \Big| = o(\|\na \phi\|_n)_{j \to +\infty}
\end{equation}
for any $\phi \in \Wnr$. Making use of $(F'_2)$, it is routine to see from \eqref{eq:PS-1} and \eqref{eq:PS-2} that $(u_j)_j$ is bounded in $\Wnr$. Consequently, there hold
\[
|\nabla u_j|^{n-2}\nabla u_j \in [L_{n/(n-1)}(B)]^n, \quad F(x,(u_j)_+) \in L^1(B), \quad f(x,(u_j)_+)u_j \in L^1(B).
\]
Passing to a subsequence, if necessary, there is some $u \in \Wnr$ such that
\begin{itemize}
 \item $u_j \rightharpoonup u$ in $\Wnr$,
 \item $u_j \to u$ in $L^p(B)$ for any $p \geq 1$, and
 \item $u_j \to u$ a.e. in $B$.
\end{itemize} 
Now we can apply a general $L^1$-convergence originally due to de Figueiredo, Miyagaki, and Ruf \cite{FMR, doO1996} to conclude that
\[
f(x,(u_j)_+) \to f(x,u_+)
\]
in $L_1 (B)$. Now we shift to the second convergence of the lemma. From \eqref{eq:growthu}, we have 
\[
|u_j(r)| \leq \Big(- \frac n{\al_n} \log r \Big)^{\frac{n-1}n} \|\na u_j\|_n,
\]
which implies $(u_j)_j$ is uniformly bounded in $B\setminus B_\epsilon(0)$ for any $\epsilon >0$. Since
\[
f(x,(u_j)_+) u_j \to f(x,u_+) u
\]
a.e. in $B$, the Lebesgue dominated convergence theorem implies that
\begin{equation}\label{eq:limfu}
\lim_{j\to \infty} \int_{B\setminus B_\epsilon(0)}|f(x,(u_j)_+) u_j - f(x,u_+) u| dx =0
\end{equation}
for any $\ep >0$. Repeating the proof of {\bf Assertion $2$} in \cite{doO1996} and using \eqref{eq:limfu}, we obtain
\[
\lim_{j\to \infty} \int_{B\setminus B_{\epsilon}(0)} \big(|\na u_j|^{n-2} \na u_j - |\na u|^{n-2} \na u\big)\big(\na u_j - \na u\big) dx =0,
\]
for any $\epsilon >0$. Hence $\na u_j \to \na u$ a.e. in $B$. Since $|\na u_j|^{n-2} \na u_j$ is bounded in $(L_{n/(n-1)}(B))^n$, we have 
\[
|\na u_j|^{n-2} \na u_j \to |\na u|^{n-2} \na u
\]
in $(L_p(B))^n$ for any $1< p < \frac n{n-1}$. This concludes the second convergence of the lemma. The proof is complete.
\end{proof}


\subsection{Proof of Theorem \ref{thmApplication}}

We are now in position to prove Theorem \ref{thmApplication}. Indeed, in view of Lemmas \ref{lem-I-AtInfinity} and \ref{lem-I-SomeWhere}, the functional $I$ has a mountain-pass geometry. Hence, we can apply a well-known Mountain-Pass lemma due to Brezis--Nirenberg \cite{BN} to conclude that there exists a Palais--Smale sequence $(u_j)_j \subset \Wnr$ of some level $\mathscr C$. Then, Lemma \ref{lem-I-Convergence} and \eqref{eq:PS-2} tell us that there is a weak solution $u_\infty$ of the equation
\begin{equation}\label{eq:Super*}
\left\{
\begin{aligned}
-\Delta_n u_\infty &= f(x,(u_\infty)_+) && \mbox{ in } B,\\
 u_\infty &= 0 && \mbox{ on } \pa B,
\end{aligned}
\right.
\end{equation}
In this sense, we conclude from \eqref{eq:PS-1} that
\begin{equation}\label{eq:Q}
\frac 1n \int_B |\nabla u_j|^n dx - \int_B F(x,(u_j)_+) dx \to \mathscr C
\end{equation}
as $j \to +\infty$. From ($F_2$), Lemma \ref{lem-I-Convergence}, and the Lebesgue dominated convergence theorem, we have
\[
\lim_{j\to \infty} \int_B F(x,(u_j)_+) dx = \int_B F(x,(u_\infty)_+) dx,
\]
which after passing \eqref{eq:Q} to the limit gives
\[
\lim_{j \to +\infty} \int_B |\nabla u_j|^n dx =n \Big( \mathscr C + \int_B F(x,(u_\infty)_+) dx \Big).
\]
Suppose that $u_\infty\leq 0$. In this sense and thanks to Lemma \ref{lem-I-PS}, we get
\[
\lim_{j \to +\infty} \int_B |\nabla u_j |^n = n \mathscr C < \Big( \frac{\alpha_n}{\alpha_0} \Big)^{n-1}. 
\]
Hence, we may fix some small $\varepsilon >0$ in such a way that
\[
n\mathscr C - \varepsilon \leq \|\na u_j\|_n^n \leq n\mathscr C + \varepsilon < \Big( \frac{\alpha_n}{\alpha_0} \Big)^{n-1}
\]
for any large $j$. Hence there is some $p>1$ close enough to $1$ such that
\[
p \alpha_0 \|\na u_j\|_n^{n/(n-1)} < \alpha_n
\]
for large $j$. For any $r \in (0,1)$, we have
\begin{align*}
\int_B |f(x,(u_j)_+)|^p dx & = \int_{B_r(0)} |f(x,(u_j)_+)|^p dx + \int_{B\setminus B_r(0)} |f(x,(u_j)_+)|^p dx\\
&\leq C_{p\alpha_0}^p \int_{B_r(0)} \exp(p \alpha_0 |u_j|^{\frac{n}{n-1} + |x|^\alpha}) dx + \int_{B\setminus B_r(0)} |f(x,(u_j)_+)|^p dx\\
&=C_{p \alpha_0}^p \int_{B_r(0)} \exp \Big(p \alpha_0 \| \na u_j\|_n^{\frac{n}{n-1} + |x|^\alpha} \big|\frac {u_j}{\|\na u_j\|_n} \big|^{\frac{n}{n-1} + |x|^\alpha} \Big) dx \\
&\quad + \int_{B\setminus B_r(0)} |f(x,(u_j)_+)|^p dx\\
&\leq C_{p \alpha_0}^p \int_{B_r(0)}\exp \Big(p \alpha_0 \|\na u_j\|_n^\frac{n}{n-1} \big(\frac{\al_n}{p\al_0}\big)^{\frac{n-1}n |x|^\alpha} \big|\frac {u_j}{\|\na u_j\|_n} \big|^{\frac{n}{n-1} + |x|^\alpha} \Big) dx \\
&\quad + \int_{B\setminus B_r(0)} |f(x,(u_j)_+)|^p dx.
\end{align*}
Notice that 
\[
\lim_{r \to 0} \sup_{B_r(0)} \big(\frac{\al_n}{p\al_0}\big)^{\frac{n-1}n |x|^\alpha} = 1.
\]
So we can choose $r \in (0,1)$ such that 
\[
p \alpha_0 \| \nabla u_j\|_n^\frac{n}{n-1} \big(\frac{\al_n}{p\al_0}\big)^{\frac{n-1}n |x|^\alpha} \leq \al_n
\]
for any $|x| \leq r$. On the other hand, there exists $R$ such that $|u_j| \leq R$ on $B\setminus B_r$ for any $j$ (this is derived from \eqref{eq:growthu}). Hence,
\[
\int_B |f(x,(u_j)_+)|^p dx \leq C_{p \alpha_0}^p \int_{B}\exp \Big(\al_n \big|\frac {u_j}{\|\na u_j\|_n} \big|^{\frac{n}{n-1} + |x|^\alpha} \Big) dx + |B| \sup_{B \times [0,R)} |f(x,t)|^p.
\]
We again apply our inequality \eqref{eq:SuperMT2} to get
\[
\sup_{j} \int_B |f(x,(u_j)_+)|^p dx < +\infty.
\]
Since $f(x,(u_j)_+) \to 0$ a.e. in $B$, there holds
\[
\lim_{j\to \infty} \int_B |f(x,(u_j)_+)|^q dx =0
\]
for any $1< q < p$. Keep in mind that $\|u_j\|_s \leq C \|\na u_j\|_n$ for any $n < s< +\infty$. Using H\"older's inequality and the previous limits, we easily see that
\[
\lim_{j\to \infty} \int_{B} f(x,(u_j)_+) u_j dx =0.
\]
Using this limit and \eqref{eq:PS-2} with $\phi$ replaced by $u_j$ we deduce that
\[
u_j \to 0
\]
strongly in $\Wnr$, which is a contradiction because $n \mathscr C>0$. Thus $(u_\infty)_+ \not\equiv 0$. 

Since the regularity of $u_\infty$ follows from standard arguments, it remains to check $u_\infty >0$ in $B$. Indeed, let
\[
g(x) =f(x,(u_\infty)_+) \geq 0.
\] 
Clearly, $g$ is a radially symmetric function and $g\geq 0$. (Hence, we do not distinguish $g(x)$ and $g(|x|)$.) Notice that $g\not\equiv 0$ since otherwise we would have $-\De_n u_\infty =0$ in $B$ and $u_\infty =0$ on $\pa B$. A standard argument implies that $u_\infty \equiv 0$, which is a contradiction. Hence, this and the continuity and non-negativity of $g$ imply that there is some $r_0 \in (0,1]$ such that $\int_0^r g(s) s^{n-1} ds >0$ for any $r \in (r_0,1)$. However, making use of \eqref{eq:Super*} we know that that 
\[
-|u_\infty'(r)|^{n-2} u_\infty'(r) = \frac1{r^{n-1}} \int_0^r g(s) s^{n-1} ds,
\]
for any $r \in (0,1)$. Hence, on one hand, we have $u_\infty'(r) \leq 0$ for any $r \in (0,1)$, however, on the other hand we further have $u_\infty'(r) < 0$ for any $r \in (r_0,1)$. Putting these facts together, we deduce that $u_\infty(r) > 0$ for any $r \in (0,1)$. Thus, $u_\infty$ is indeed a positive solution to \eqref{eq:Super}.



\section*{Acknowledgments}

The research of Q.A.N is funded by the Vietnam National University, Hanoi (VNU) under project number QG.19.12. The research of V.H.N is partially funded by the Simons Foundation Grant Targeted for Institute of Mathematics, Vietnam Academy of Science and Technology. 


\addtocontents{toc}{\protect\setcounter{tocdepth}{0}}

\section*{ORCID iDs}

\noindent Qu\cfac oc Anh Ng\^o: 0000-0002-3550-9689

\noindent Van Hoang Nguyen: 0000-0002-0030-5811

\addtocontents{toc}{\protect\setcounter{tocdepth}{2}}


\end{document}